\documentclass[manyauthors]{fundam}
\usepackage{url} 
\usepackage[ruled,lined]{algorithm2e}
\usepackage{graphicx}

\usepackage{amssymb, amsmath, graphicx, enumerate,enumitem}
\usepackage{amstext,amsopn,amsfonts,eucal,amssymb}
\usepackage{graphicx, url, float, multirow, chngpage, color}
\usepackage{hyperref,graphicx,wrapfig,caption,comment} 
\usepackage[percent]{overpic}

\def\mc{\mathcal}

\usepackage[caption=false]{subfig}
\usepackage{environ}
\usepackage{color}
\usepackage{tabu}
\usepackage{mathtools}
\usepackage{rotating}

\DeclareMathOperator{\pre}{pre}
\DeclareMathOperator{\suf}{suf}
\DeclareMathOperator{\Rep}{Rep}
\DeclareMathOperator{\Ret}{Ret}

\newcommand{\uRed}[1]{\textcolor{red}{\underline{#1}}}

\begin{document}
\title{Insertions Yielding Equivalent Double Occurrence Words}
\author{Daniel A. Cruz\\
University of South Florida, Tampa FL, 33620, USA\\
dcruz@mail.usf.edu
\and Margherita Maria Ferrari\\
University of South Florida, Tampa FL, 33620, USA\\
mmferrari@usf.edu
\and Nata\v sa  Jonoska\\
University of South Florida, Tampa FL, 33620, USA\\
jonoska@mail.usf.edu
\and Lukas Nabergall\\
University of Waterloo, Waterloo ON, N2L 3G1, Canada\\
lnaberga@uwaterloo.ca
\and Masahico Saito\\
University of South Florida, Tampa FL, 33620, USA\\
saito@usf.edu
}
\address{Department of Mathematics and Statistics, University of South Florida, Tampa FL, 33620, USA}
\maketitle

\runninghead{D.A. Cruz, M.M. Ferrari, N. Jonoska, L. Nabergall, M. Saito}{Insertions on Double Occurrence Words}

\begin{abstract}
A double occurrence word (DOW) is a word in which every symbol appears exactly twice; two DOWs are equivalent if one is a symbol-to-symbol image of the other.
We consider the so called repeat pattern ($\alpha\alpha$) and the return pattern ($\alpha\alpha^R$), with gaps allowed between the $\alpha$'s.
These patterns generalize square and palindromic factors of DOWs, respectively.
We introduce a notion of inserting repeat/return words into DOWs and study how two distinct insertions into the same word can produce equivalent DOWs.
Given a DOW $w$, we characterize the structure of $w$ which allows two distinct insertions to yield equivalent DOWs.
This characterization  depends on the locations of the insertions and on the length of the inserted repeat/return words and implies that when one inserted word is a repeat word and the other is a return word, then both words must be trivial (i.e., have only one symbol).
The characterization also introduces a method to generate families of words recursively.
 \end{abstract}

\section{Introduction}\label{s_intro}

A word $w$ over an alphabet $\Sigma$ is a double occurrence word (DOW) if each element of $\Sigma$ appears either zero or two times.
DOWs have been studied in relation to knot theory \cite{Burns2015,Gibson2011,Turaev2004}, mathematical logic \cite{Courcelle2008}, and algebraic combinatorics \cite{Shtylla2009}. 
DOWs are also known as Gauss words and are closely related to linear diagrams, chord diagrams, and circle graphs. 
In the context of genomics, DOWs and operations on DOWs have been used in studies of DNA rearrangement \cite{Braun2018,Burns2013,Ehrenfeucht2003,Jonoska2017}.
By modeling the DNA rearrangement process using DOWs, it was observed that over 95\% of the scrambled genome of the ciliate \textit{Oxytricha trifallax} could be described by iterative insertions of the ``repeat pattern'' ($\alpha\alpha$) and the ``return pattern'' ($\alpha\alpha^R$) \cite{Burns2016}.
Roughly speaking, a {\em pattern} is a sequence of 
variables, and we say that an instance of a pattern {\em appears in a word $w$} if each variable of the pattern can be mapped to a non-empty factor of $w$ \cite{Jonoska2017}.
The repeat pattern $\alpha\alpha$ generalizes square factors while the return pattern $\alpha\alpha^R$ generalizes palindromic factors.
We refer to instances of the repeat and return pattern as {\em repeat words} and {\em return words}, respectively \cite{Arredondo2014,Burns2016,Jonoska2017}.

Patterns in DNA rearrangement are discussed in \cite{Burns2016}, while transformations on DOWs where instances of patterns are deleted 
or inserted are considered in \cite{Jonoska2017}. 
In studies of DNA rearrangement, it has been observed that the insertion of a repeat or return pattern may have evolutionary significance~\cite{Chang2005}, and the process of obtaining one word from another by the insertion of a repeat or return word may be of interest. Relatedly, similar operations on words have been studied with applications to computational linguistics and natural language processing. In the literature, four so-called edit operations are primarily considered: insertion of a symbol, deletion of a symbol, substitution of one symbol with another, and transposition of two adjacent symbols \cite{Levenshtein,Needleman,Sankoff}. The pattern-based word transformations considered here may be regarded as a generalization of these edit operations. 

Here we  define a notion of inserting repeat and return words in DOWs at prespecified indices. We consider equivalence classes of DOWs where two DOWs are equivalent if one 
is obtained from the other by a symbol-to-symbol morphism. 
Equivalent words correspond to the same chord diagrams, as well as isomorphic assembly graphs.
The main question considered here is under which conditions two distinct insertions into the same word can produce equivalent DOWs. 
A pair of insertions in a given DOW falls in one of the three types: interleaving, nested, and sequential (see Section \ref{s_equiv} for  definitions). The paper characterizes the words that yield the equivalent results in each of these situations.

\section{Background}\label{s_backg}

An {\em ordered alphabet} $\Sigma$ is a countable set with a linear order that is bounded below which can naturally be identified with the set of natural numbers $\mathbb{N} = \{1, 2, \ldots\}$, so we set $\Sigma = \mathbb{N}$ throughout the rest of the paper.
For $n \in \mathbb{N}$, we denote $ \{1, 2, \ldots, n\}$ by $[n]$.
For the remainder of the paper, we reserve the letters $a,b$ as {\it symbols} in $\Sigma$, and reserve the letters  $s,t,u,\ldots,z$ as ``words'' in $\Sigma^\ast$ (defined below).

We use standard definitions and conventions (e.g.,  \cite{Choffrut1997,Courcelle2008,Lothaire2002,Jonoska2017}).
A {\em word $w$ over $\Sigma$} is a finite sequence of symbols $a_1\cdots a_n$ in $\Sigma$; the  {\em length of $w$}, denoted $|w|$, is $n$.
The set of all words over $\Sigma$ is denoted by $\Sigma^*$ and  includes the empty word $\epsilon$ whose length is $0$; and  $\Sigma^+=\Sigma^*\setminus \{\epsilon\}$.
The set of all symbols $\{a_1,\ldots,a_n\}$ comprising $w$ is denoted by $\Sigma[w]$.
The {\em reverse of } $w = a_{1}a_{2}\cdots a_{n}$ ($a_{i} \in \Sigma$) is the word $w^{R} = a_{n}\cdots a_{2}a_{1}$.
The word $v$ is a {\em factor} of the word $w$, denoted $v \sqsubseteq w$, if $\exists w_1, w_2 \in \Sigma^*$ such that $w = w_1 v w_2$;  if $w_1=\epsilon$ then $v$ is  a {\it prefix} of $w$, while if $w_2=\epsilon$ then $v$ is a {\it suffix} of $w$; the word $w_1w_2$ is denoted by $w - v$. 
If $|v|<|w|$, then $v$ is a {\em proper} factor, prefix, or suffix as appropriate.
The set of common factors of $u, v \in \Sigma^{*}$ is denoted by $u \cap v$.
Let $1\leq d\leq n=|w|$, and write $w=a_1a_2 \cdots a_n$.
The prefix (suffix resp.) of $w$ with length $d$ is denoted $\pre(w,d)=a_1a_2\cdots a_d$ ($\suf(w,d)= a_{n-d+1} \cdots a_{n-1}a_n$ resp.).

A word $w$ in $\Sigma^\ast$ is a {\em double occurrence word} (DOW) if every symbol in $w$ appears either zero or twice.
We use $\Sigma_{DOW}$ to denote the set of all double occurrence words over $\Sigma$.
{\em Single occurrence words} (SOWs) denoted $\Sigma_{SOW}$ are nonempty words with distinct symbols.
Given $w\in\Sigma_{DOW}$, $|w|/2$ is the {\em size of $w$}.
Given $u\in\Sigma_{SOW}$, we say that $uu$ is a {\em repeat word in $w$} if $w = z_1 u z_2 u z_3$ for some $z_1,z_2,z_3\in\Sigma^\ast$.
Similarly, $uu^R$ is a {\em return word in $w$} if $w = z_1 u z_2 u^R z_3$ for some $z_1,z_2,z_3\in\Sigma^\ast$.

A morphism $f$ on $\Sigma^*$ induced by a bijection (symbol-to-symbol map) on $\Sigma$ is called an {\it equivalence map}.
We write $w_1\sim w_2$ if there is an equivalence map $f$ such that $f(w_1)=w_2$.
The relation $\sim$ is an equivalence on $\Sigma^\ast$.
A word $w=a_1 a_2 \cdots a_n$ is in {\em ascending order} if $a_1$ is the least element in the alphabet and the first appearance of a symbol is one greater than the largest of all preceding symbols in the word \cite{Burns2013}.
Since a word in ascending order is unique in the considered alphabet, we take words in ascending order as {\em class representatives} of the equivalence classes determined by the relation $\sim$.

\begin{example}
Consider the words $w=121323$ and $w'=131232$. Note that $w\sim w'$, but $w$ is in ascending order while the word $w'$ is not because the symbol $3$ appears before $2$.
\end{example}

\begin{definition}\label{j_insert}
Let $w=z_1z_2z_3 \in\Sigma_{DOW}$ for some $z_1, z_2, z_3 \in \Sigma^\ast$ be in ascending order.
Let $u$ be a SOW over $\Sigma\setminus\Sigma[w]$ which is in ascending order and $|u|=\nu$.
Suppose $k$ and $\ell$ are such that $k-1=|z_1|$ and $\ell-1=|z_1z_2|$. Then 
\begin{itemize}
\setlength{\itemsep}{-3pt}
\item $w' = z_1 u z_2 u z_3$ is obtained from $w$ by a {\em repeat insertion} denoted 
$w'=w\star\rho(\nu,k,\ell)$, and
\item $w' = z_1 u z_2 u^R z_3$ is obtained from $w$ by a {\em return insertion} denoted 
$w'=w\star\tau(\nu,k,\ell)$.
\end{itemize}
\end{definition}

We do not specify the word $u$ in the notation of repeat and return insertions because the inserted word has distinct symbols from $w$ and consists of symbols immediately following the largest symbol of $w$; hence it is uniquely determined by its length.

\begin{example} \label{insert_ex}
Let $w = 1232314554$, then
\begin{align*}
w\star \rho(2,4,6) &= 123\uRed{67}23\uRed{67}14554 = w_1 \qquad\qquad w\star\tau(2,7,11) = 123231\uRed{67}4554\uRed{76} = w_3\\
w\star \rho(2,2,4) &= 1\uRed{67}23\uRed{67}2314554 = w_2 \qquad\qquad\,\,\, w\star\tau(2,9,9) = 12323145\uRed{6776}54 = w_4
\end{align*}
Observe that $w_1\sim w_2$ and $w_3\sim w_4$ but $w_1\not\sim w_3$.
\end{example}

Let $v$ be a repeat or return word in $w\in\Sigma_{DOW}$; we write $v=uu'$ where $u'=u$ if $v$ is a repeat word and $u'=u^R$ if $v$ is a return word.
For the rest of the paper, we use $u'$ for a SOW $u$ to denote $u$ or $u^R$ as is appropriate in the context of repeat and return words.
We use the notation $w\star \mathcal{I}(\nu,k,\ell)$ to indicate that the insertion is either a repeat insertion or a return insertion.
We denote by {\em Rep} ({\em Ret} resp.) the set of all repeat insertions (return insertions resp.), and we write $\mathcal{I}\in\Rep$ ($\mathcal{I}\in\Ret$ resp.) to indicate that the insertion is a repeat insertion (return insertion resp.).
Observe that $\left(w\star\mathcal{I}(\nu,k,\ell)\right)^R \sim w^R\star\mathcal{I}(\nu,|w|-\ell+2,|w|-k+2)$.
We say that $uu$ or $uu^R$ is a {\em trivial} repeat or return word, respectively, if $|u|=1$, and an insertion $\mathcal{I}(1,k,\ell)$ is called {\it trivial}. A trivial repeat insertion into $w$ is also a trivial return insertion, so we focus on nontrivial insertions.
As a convention, if $k=\ell$ we set
$w \star\mathcal{I}(\nu,k,k) =z_1 uu' z_2z_3$ as shown in $w_4$ in Example~\ref{insert_ex}.

\begin{remark}\label{lem_equal-ins}
Suppose $w\star \mc I_1(\nu_1,k_1,\ell_1)\sim w\star  \mc I_2(\nu_2,k_2,\ell_2)$, then $\nu_1=\nu_2=\nu$, and because the inserted word has distinct symbols from $w$, by definition we have that $u_1u_1'$ inserted by $\mc I_1$ and $u_2 u_2'$ by $\mc I_2$ are such that $u_1=u_2=u$  (and $u_1', u_2'\in \{u,u^R\}$).  
For all $w'\sim w$ we also have that 
$w'\star \mc I_1(\nu,k_1,\ell_1)\sim w'\star  \mc I_2(\nu,k_2,\ell_2)$. This follows from the fact that if $f(w\star \mc I_1(\nu,k_1,\ell_1)) = w\star  \mc I_2(\nu,k_2,\ell_2)$ and $g(w')=w$ for equivalence maps $f$ and $g$, 
then $f(g(w')\star \mc I_1(\nu,k_1,\ell_1))=g(w')\star  \mc I_2(\nu,k_2,\ell_2)$.
\end{remark}

Considering the above remark, we have the following definition for equality of two insertions.

\begin{definition}\label{j_equal}
We say that two nontrivial insertions $\mc I_1(\nu,k_1,\ell_1)$ and  $\mc I_2(\nu,k_2,\ell_2)$ are {\em equal  for a DOW $w$} if 
$w\star \mc I_1(\nu,k_1,\ell_1)\sim w\star  \mc I_2(\nu,k_2,\ell_2)$ and $(k_1,\ell_1)=(k_2,\ell_2)$.
\end{definition}

It becomes a natural question to consider the situations when two distinct (unequal) nontrivial insertions into $w$ 
yield equivalent words. 

The following three results are used repeatedly throughout Section~\ref{s_equiv}.
The first generalizes a well known lemma by Lyndon and Sch\"utzenberger \cite{L-S}; it describes the structures of two equivalent words where a suffix of one word is a prefix of the other.

\begin{lemma}\label{lem_L-S}
Let $s\in\Sigma^+$ and $t,z\in\Sigma^\ast$ such that $f(sz)=zt$ for an equivalence map $f$. 
Then, $s=s_1s_2$ with $z=f(s)\cdots f^{h-1}(s)f^h(s_1)$, and $t=f^h(s_2)f^{h+1}(s_1)$ for $h=\left\lceil|z|/|s|\right\rceil$.
\end{lemma}

\begin{figure}[h!]
  \begin{minipage}{.3\textwidth}
  \centering
  \bf{(a)}\includegraphics[scale=.9]{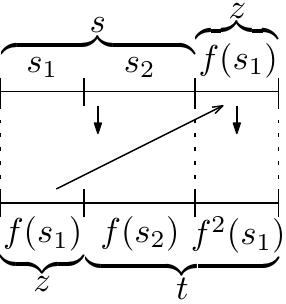}
  \end{minipage}
  \hfill
  \begin{minipage}{.6\textwidth}
  \centering
  \bf{(b)}\includegraphics[scale=.9]{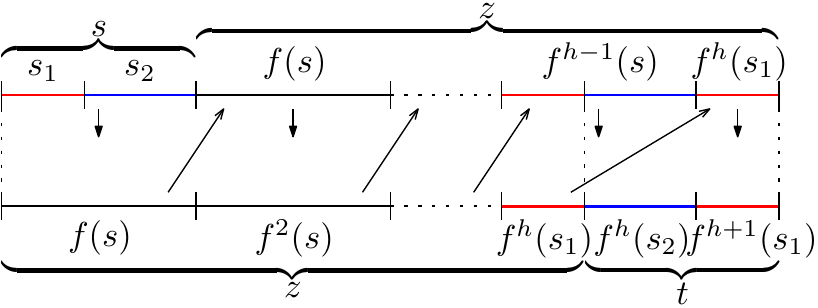}
  \end{minipage}
  \caption{Representation of $sz\sim zt$ when \textbf{(a)} $0<|z|<|s|$, or \textbf{(b)} $|z|\geq |s|$.}
  \label{L-S}
\end{figure}

\begin{proof} If $z=\epsilon$, then $f(t)=s$ and the statement holds with $h=0$ and $s=s_1$.
If $0<|z|<|s|$, then write $s=s_1s_2$ where $|s_1|=|z|$.
It follows that $f(s_1)=z$ because $f(s_1)$ is a prefix of $f(s)$, and $f(s_2)$ is a prefix of $t$ because $f(s_2)$ is a suffix of $f(s)$ and $f(s)\cap t\neq\emptyset$.
Then $sz\sim zt$ implies $f(z)=f^2(s_1)$ is a suffix of $t$, so $t=f(s_2)f^2(s_1)$ (Figure \ref{L-S}a).
If $|z|\geq|s|$, let $h=\left\lceil|z|/|s|\right\rceil$ and write $s=s_1s_2$ where $|s_1|=|z|-(h-1)|s|$.
Note that $s_1\neq\epsilon$.
Then $z=f(s)\cdots f^{h-1}(s)f^h(s_1)$ and $t=f^h(s_2)f^{h+1}(s_1)$ (Figure \ref{L-S}b).
\end{proof}

In Lemma~\ref{lem_L-S}, note that if $sz=zt$, then $f$ is the identity mapping, and so we arrive at the lemma from \cite{L-S}. In Section~\ref{s_equiv}, we use the lemma symmetrically when $s=t$, and $sz\sim zs$. An illustration of the situation described with Lemma~\ref{lem_L-S} can be seen within the prefixes of $w_1$ and $w_2$, as well as in the suffixes of $w_3$ and $w_4$.

\begin{lemma}\label{lem_short} Suppose $x,y_1,y_2\in \Sigma^+$ with $|y_1|=|y_2|$ and $f$ is an equivalence map.
Let $f(xy_1)=y_1x$ and either $f(x^Ry_2)=y_2x^R$ or  $f(y_2x)=xy_2$. 
Then either $|x|\le |y_1|=|y_2|$ or  $\Sigma[x]\cap \Sigma[y_2]\not=\emptyset$.
\end{lemma}

\begin{proof}
Let $x=x'a$ for $a\in \Sigma$ and let $|x|>|y_1|=|y_2|$. Then  $f(x'ay_1)=y_1x$  implies $f(a)\in \Sigma[x]$. In both cases when 
$f(ax'^Ry_2)=y_2x^R$  or $f(y_2x'a)=xy_2$ we have $f(a)\in \Sigma[y_2]$.
\end{proof}

The last lemma from \cite{Jonoska2017} details useful properties of repeat and return words. 

\begin{lemma}\label{lem_comp}
Let $w \in \Sigma_{DOW}$, let $xx'$ and $yy'$ be two repeat (return resp.) words in $w$, and let $u \in x \cap y$.
Then both $(x - u)(x' - u)$ ($(x - u)(x' - u^R)$ resp.) and $(y - u)(y' - u)$ ($(y - u)(y' - u^R)$ resp.) are repeat (return resp.) words in $w$.
Furthermore, if $xx'$ and $yy'$ are repeat and return words, respectively, in $w$, then $x \cap y = \{\epsilon\}$, $|x| = 1$, or $|y| = 1$. 
\end{lemma}

\section{Insertions Yielding Equivalent Words}\label{s_equiv}

In this section, we fix $w\in\Sigma_{DOW}$ in ascending order of length $n$.
Let $\mathcal{I}_i(\nu_i,k_i,\ell_i)$ for $i=1,2$ be distinct insertions into $w$
such that  $w_i=w\star\mathcal{I}_i(\nu_i,k_i,\ell_i)$ are equivalent with $f(w_1)=w_2$ for an equivalence map $f$.
By Remark~\ref{lem_equal-ins} and Definition~\ref{j_equal}, $\nu_1=\nu_2=\nu$ and  $(k_1,\ell_1)\neq(k_2,\ell_2)$.
Without loss of generality, we assume that $k_1\leq k_2$. Because the inserted words have no symbols in common with $w$, we consider that $uu_1'$ is inserted with $\mc I_1$ and  $uu_2'$ is inserted with $\mc I_2$ ($u_1',u_2'\in \{u,u^R\}$).
If $\mathcal{I}_1,\mathcal{I}_2\in\Rep$ ($\mathcal{I}_1,\mathcal{I}_2\in\Ret$ resp.), then we consider that they both insert $uu'$.

Observe that $k_1\neq k_2$, because  if $k_1=k_2$, then $f(u)=u$, implying 
$\ell_1=\ell_2$, and hence the insertions are equal.
Up to symmetry, the indices $k_1,k_2,\ell_1,\ell_2$ can have the following possibilities.
\begin{itemize}
\setlength{\itemsep}{-3pt}
\item {\em Interleaving} insertions, $k_1<k_2\leq \ell_1<\ell_2$: $\mc I_2$ inserts $u$ at a location before $\mc I_1$ inserts $u_1'$.
\item {\em Nested} insertions, $k_1<k_2\leq \ell_2<\ell_1$: $\mc I_2$ inserts $u$ and $u_2'$ at locations before $\mc I_1$ inserts $u_1'$.
\item {\em Sequential} insertions, $k_1\leq\ell_1<k_2\leq\ell_2$: $\mc I_2$ inserts $u$ and $u_2'$ at locations after  $\mc I_1$ inserts $u_1'$. 
\end{itemize}

Further, without loss of generality, we can assume $k_1=1$ and $\ell=\max\{\ell_1,\ell_2\} = n+1$.
For the remainder of the section, we set $w=z_1z_2z_3$ where $|z_1|=k_2-1$ in the case of interleaving and nested 
insertions and $|z_1|= \ell_1-1$ in the case of sequential insertions. Also $|z_1z_2|$ equals  $\ell_1-1$, $\ell_2-1$, or  
$k_2-1$ for interleaving, nested, or sequential insertions, respectively.
Hence the four positions for the two insertions into $w$ are 1, $|z_1|+1$, $|z_1z_2|+1$, and $n+1$.
We consider that $w$ is in ascending order.

\subsection{Interleaving and Nested Insertions} \label{ss_A+B}
Let $t\in\Sigma_{DOW}$ be in ascending order.
We say that $t$ is an {\em interleaving sequence of return words} if there exist integers $h,\nu\geq 1$ such that $t=x_1x_2\cdots x_hx_1^Rx_2^R\cdots x_h^R$, where $|x_i|=\nu$ for all $i$; in this case we write $t=\text{Int}(h,\nu)$.
Clearly, $x_i\in\Sigma_{SOW}$ for all $i$.
Note that if $\nu=1$, then $\text{Int}(h,1)$ is a repeat word of size $h\geq 1$.

In Proposition~\ref{prop_caseA}, we consider two interleaving insertions into a DOW $w$ which yield equivalent words.
To describe the structure of $w$, we track the image of the inserted word $u$ in the resulting words $w_1,w_2\in\Sigma_{DOW}$ by using the equivalence map as described by Lemmas \ref{lem_L-S} and \ref{lem_short}.
We also use this ``image tracking'' method when considering nested and sequential insertions.

\begin{proposition}\label{prop_caseA}
If  $\mc I_1$ and $\mc I_2$ are interleaving insertions, that is $k_1 < k_2 \leq \ell_1 < \ell_2$, 
 then $z_1,z_3\in\Sigma_{SOW}$ are such that 
\begin{enumerate}[label={(\arabic*)}]
\setlength{\itemsep}{-3pt}
\item If $\mathcal{I}_1,\mathcal{I}_2\in\Rep$, then $z_1z_3$ is a repeat word in $w$.
\item If $\mathcal{I}_1,\mathcal{I}_2\in\Ret$, then $z_1z_3\sim \text{Int}(h,\nu)$ where $h=(k_2-k_1)/\nu$ is a positive integer.
\end{enumerate}
\end{proposition}

\begin{proof} We have that $w_1=uz_1z_2u_1'z_3$ and $w_2=z_1uz_2z_3u_2'$. Because $uz_1\sim z_1u$ and $u$ is SOW, we have that $z_1$ is a SOW by Lemma~\ref{lem_L-S}.
Similarly, we have that $z_3$ is a SOW, and thus $|z_1|=|z_3|$ since $z_2$ is a DOW.
Moreover, $z_1=x_1x_2\cdots x_h$, where $f^i(u)=x_i$ for $1\leq i < h$ (if any such exist) and $x_h=f^h(v_1)$ where $v_1$
is a prefix of $u$. 
\\

(1) Suppose $\mathcal{I}_1,\mathcal{I}_2\in\Rep$. Then $u=u_1'=u_2'$ and the two equivalences $uz_1\sim z_1u$ and $uz_3\sim z_3u$ are with the same equivalence map. 
Since $z_1$ and $z_3$ are of the same length, by Lemma \ref{lem_L-S} they have the same structure,
i.e., $z_1=x_1x_2\cdots x_h=z_3$.\\

(2) Suppose $\mathcal{I}_1,\mathcal{I}_2\in\Ret$. Then $u_1'=u_2'=u^R$. By Lemma~\ref{lem_short}, $|z_1|=|z_3|\geq|u|$ because $f(uz_1)=z_1u$, $f(u^Rz_3)=z_3u^R$, and $\Sigma[w]\cap \Sigma[u]=\emptyset$. Finally, observe  that $|x_h|=|u|=\nu$.

\begin{figure}[h!]
  \centering
  \includegraphics[]{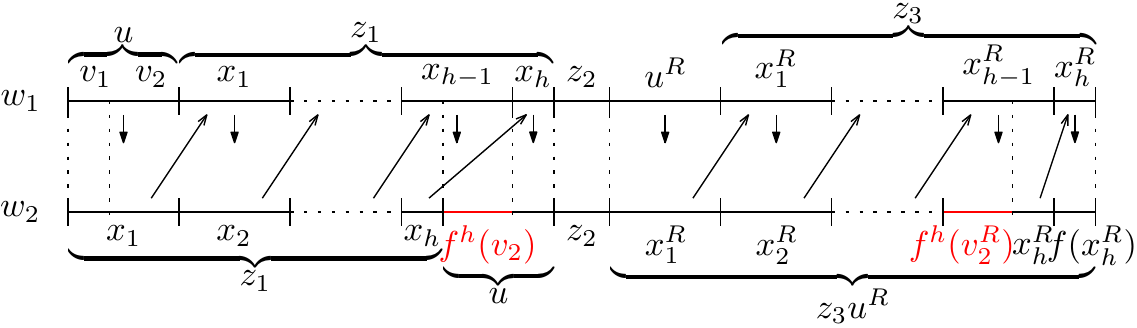}
  \caption{Representation of $w_1$ and $w_2$ when $\mathcal{I}_1,\mathcal{I}_2\in\Ret$ are such that $k_1<k_2\leq\ell_1<\ell_2$.}
  \label{A-RetRet} 
\end{figure}

 Since $u^Rz_3 \sim z_3u^R$, by Lemma~\ref{lem_L-S},
 $z_3=x_1^Rx_2^R\cdots x_{h-1}^Rx_h^R$ with $f^i(u)=x_i$ ($1\le i< h$) and $u=f^{h}(v_2)f^{h+1}(v_1)=v_1v_2$.
It follows that $f(x_{h-1}^Rx_h^R)=f^{h}(v_2^R)f^{h}(v_1^R)f^{h+1}(v_1^R) = f^{h}(v_2^R)x_h^Rf^{h+1}(v_1^R)$ is a suffix of $z_3u^R$ (Figure~\ref{A-RetRet}).
If $|x_h|=|v_1|<|u|$, $f^{h+1}(v_1^R)$ is a proper suffix of $u^R$, and so $\emptyset\neq\Sigma[x_h^R]\cap\Sigma[u]\subset
\Sigma[z_3]\cap \Sigma[u]$, contradictory to Definition~\ref{j_insert}.
The above implies that $h=|z_1|/|u|=(k_2-k_1)/\nu$ is a positive integer.
\end{proof}

\begin{example}\label{ex_a}
Let $w=12345677612345$ and consider the insertions $\rho(2,1,10)$ and $\rho(2,6,15)$ into $w$; also let $w'=123456652143$ with insertions $\tau(2,1,9)$ and $\tau(2,5,13)$ into $w'$.
\begin{align*}
w_1&=\uRed{89}123456776\uRed{89}12345\qquad\qquad 
w_1'=\uRed{78}12345665\uRed{87}2143\\
w_2&=12345\uRed{89}677612345\uRed{89}\qquad\qquad 
w_2'=1234\uRed{78}56652143\uRed{87}
\end{align*}
Note that $w_1\sim w_2$ with $z_1=z_3=12345$, and $w_1'\sim w_2'$ with $z_1z_3=12342143\sim\text{Int}(2,2)$.
\end{example}

Let $t\in\Sigma_{DOW}$ be in ascending order.
We say that $t$ is a {\em nested sequence of repeat words} if there exist integers $h,\nu\geq 1$ such that $t=x_1x_2\cdots x_hx_hx_{h-1}\cdots x_1$, where $|x_i|=\nu$ for all $i$; in this case we write $t=\text{Nes}(h,\nu)$.
Clearly, $x_i\in\Sigma_{SOW}$ for all $i$ and if $\nu=1$, then $\text{Nes}(h,1)$ is a return word of size $h\geq 1$.

\begin{proposition}\label{prop_caseB}
If  $\mc I_1$ and $\mc I_2$ are nested insertions, that is $k_1 < k_2 \leq \ell_2 < \ell_1$, 
 then $z_1,z_3\in\Sigma_{SOW}$ are such that 
 \begin{enumerate}[label={(\arabic*)}]
 \setlength{\itemsep}{-3pt}
\item If $\mathcal{I}_1,\mathcal{I}_2\in\Ret$, then $z_1z_3$ is a return word in $w$.
\item If $\mathcal{I}_1,\mathcal{I}_2\in\Rep$, then $z_1z_3\sim \text{Nes}(h,\nu)$ where $h=(k_2-k_1)/\nu$ is a positive integer.
\end{enumerate}
\end{proposition}
\begin{proof}
We have that $w_1=uz_1z_2z_3u_1'$ and $w_2=z_1uz_2u_2'z_3$. 
As in the proof of  Proposition~\ref{prop_caseA}, it must be 
that $z_1$ and $z_3$ are SOW by Lemma~\ref{lem_L-S} because $u$ is SOW; hence $|z_1|=|z_3|$.
Also, $z_1=x_1x_2\cdots x_h$, where $f^i(u)=x_i$ for $1\leq i < h$ (if any such exist) and $x_h=f^h(v_1)$ where $v_1$
is a prefix of $u$. 
\\

(1) First, assume that $\mathcal{I}_1,\mathcal{I}_2\in\Ret$ and $u_1'=u_2'=u^R$. Because $u^Rz_3=f(z_3u^R)$, 
applying  Lemma \ref{lem_L-S} symmetrically   we have that $z_3=x_h^Rx_{h-1}^R\cdots x_1^R=z_1^R$.\\

\begin{figure}[h!]
  \centering
  \includegraphics[]{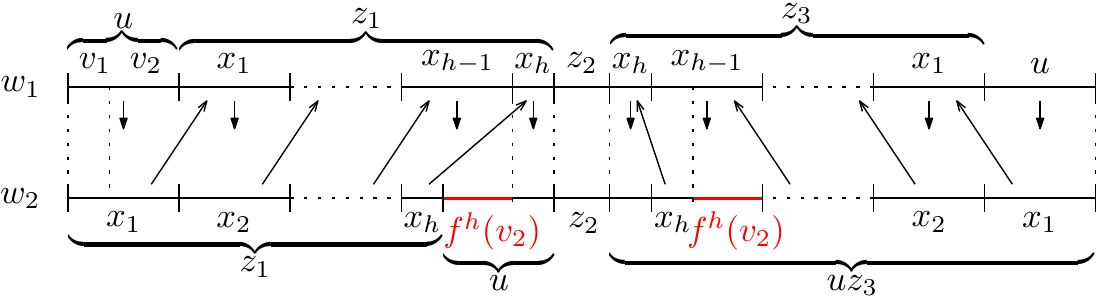}
  \caption{Representation of $w_1$ and $w_2$ when $\mathcal{I}_1,\mathcal{I}_2\in\Rep$ are such that $k_1<k_2\leq\ell_2<\ell_1$.}
  \label{B-RepRep} 
\end{figure}

(2) Suppose $\mathcal{I}_1,\mathcal{I}_2\in\Rep$ and $u=u_1'=u_2'$. Because $f(uz_1)=z_1u$ and $f(z_3u)=uz_3$, by 
Lemma~\ref{lem_short} we have $|u|\le |z_1|=|z_3|$. Similarly as in Proposition~\ref{prop_caseA}, 
we observe that $|x_h|=|u|$.
Applying  Lemma \ref{lem_L-S} symmetrically  (Figure \ref{B-RepRep}) $z_3=x_hx_{h-1}\cdots x_2x_1$.
It follows that $f(x_hx_{h-1})=f^{h+1}(v_1)f^{h}(v_1)f^{h}(v_2) = f^{h+1}(v_1)x_hf^{h}(v_2)$ is a prefix of $uz_3$
and if  $|x_h|=|v_1|<|u|$, then $f^{h+1}(v_1)$ is a proper prefix of $u$, and so $\Sigma[x_h]\cap\Sigma[u]\neq\emptyset$.
But then $\emptyset\neq\Sigma[x_h]\cap\Sigma[u]\subset
\Sigma[z_3]\cap \Sigma[u]$, contradictory to Definition~\ref{j_insert}.
This implies that $h=|z_1|/|u|=(k_2-k_1)/\nu$ is a positive integer.
\end{proof}

\begin{example}\label{ex_b}
Let $w=123456653412$ with the  insertions $\rho(2,1,13)$ and $\rho(2,5,9)$ into $w$; also let
$w'=12345677654321$ with the insertions $\tau(2,1,15)$ and $\tau(2,6,10)$ into $w'$.
\begin{align*}
w_1&=\uRed{78}123456653412\uRed{78}
\qquad\qquad w_1'=\uRed{89}12345677654321\uRed{98}\\
w_2&=1234\uRed{78}5665\uRed{78}3412
\qquad \qquad w_2'=12345\uRed{89}6776\uRed{98}54321
\end{align*}
Note that $w_1\sim w_2$ with  $z_1z_3=12343412\sim\text{Nes}(2,2)$ and $w_1'\sim w_2'$ with $z_1=12345$ and $z_3=z_1^R$.
\end{example}

\begin{proposition}\label{prop_AB-conv}
Let $\nu\in\mathbb{N}$, and suppose one of the following holds:
\vskip.5\baselineskip
\centerline{
\begin{tabular}{ll}
$(1)\quad t=12\cdots m 12\cdots m$ &
$(3)\quad t=\text{Nes}(h,\nu)$
\\
$(2)\quad t=\text{Int}(h,\nu)$ & 
$(4)\quad t=12 \cdots mm \cdots 21$
\end{tabular}
}
\vskip.5\baselineskip
\noindent for some $m\geq 1$ or $h\geq 1$, as appropriate.
Then there exist two distinct insertions, $\mathcal{I}_1(\nu,k_1,\ell_1)$ and $\mathcal{I}_2(\nu,k_2,\ell_2)$, such that $t\star\mathcal{I}_1(\nu,k_1,\ell_1) \sim t\star\mathcal{I}_2(\nu,k_2,\ell_2)$.
Moreover, the following table holds for this pair of insertions (based on the corresponding case):
\vskip.5\baselineskip
\centerline{
\begin{tabular}{c|cc}
& Interleaving & Nested\\
\hline
$\mathcal{I}_1,\mathcal{I}_2\in\Rep$ & (1) & (3) \\
$\mathcal{I}_1,\mathcal{I}_2\in\Ret$ & (2) & (4)
\end{tabular}
}
\vskip.5\baselineskip

\end{proposition}
\begin{proof}
Let $t\in\Sigma_{DOW}$ correspond to one of the given cases with $|t|=n$.
We define a pair of distinct insertions into $t$ as follows so that the table above holds:
\vskip.5\baselineskip
\begin{tabular}{ll}
$(1)$ - $(2)$ & $\mathcal{I}_1(\nu,1,\frac{n}{2}+1)$ and $\mathcal{I}_2(\nu,\frac{n}{2}+1,n+1)$\\
$(3)$ - $(4)$ & $\mathcal{I}_1(\nu,1,n+1)$ and $\mathcal{I}_2(\nu,\frac{n}{2}+1,\frac{n}{2}+1)$
\end{tabular}
\vskip.5\baselineskip
\noindent In all cases, let $uu'\in\Sigma_{DOW}$ be inserted into $t$.
Observe the following for each case:
\begin{enumerate}[label=(\arabic*)]
\setlength{\itemsep}{-3pt}
\item $(u12\cdots m)(u12\cdots m)\sim (12 \cdots m u)(12\cdots m u)$
\item $(u x_1 x_2 \cdots x_h)(u^R x_1^R x_2^R \cdots x_h^R) \sim (x_1 x_2 \cdots x_hu)(x_1^R x_2^R \cdots x_h^R u^R)$
\item $(u x_1 x_2 \cdots x_h)(x_h \cdots x_2 x_1 u) \sim (x_1 x_2 \cdots x_hu)(u x_h \cdots x_2 x_1)$
\item $(u12\cdots m)(m\cdots 21 u^R)\sim (12 \cdots m u)(u^Rm\cdots 21)$.
\end{enumerate}
\end{proof}

\subsection{Sequential Insertions} \label{ss_C}
In this section we consider the case when both indices of insertion $\mc I_1$ into $w$ precede both indices of the insertion $\mc I_2$ into $w$. First we observe the following lemma.
\begin{lemma}\label{lem_z3}
If $k_1 \leq \ell_1 < k_2 \leq \ell_2$, then $k_2-\ell_1\geq \nu$.
\end{lemma}
\begin{proof}
We have that $w_1=uz_1u_1'z_2z_3$ and $w_2=z_1z_2uz_3u_2'$ with $f(w_1)=w_2$. 

\begin{figure}[!h]
\begin{minipage}{.5\textwidth}
\centering
\bf{(a)}{\includegraphics[]{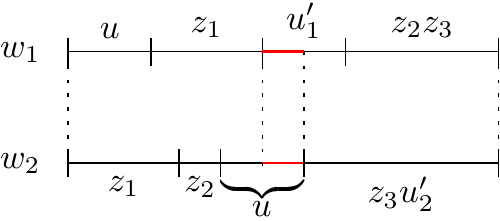}}
\end{minipage}
\hfill
\begin{minipage}{.5\textwidth}
\centering
\bf{(b)}{\includegraphics[]{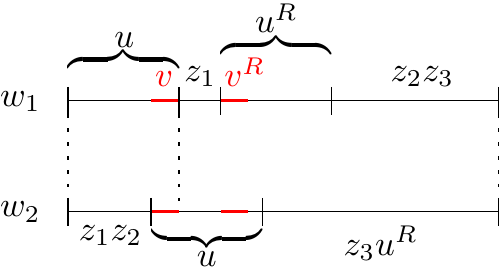}}
\end{minipage}
\caption{Representation of $w_1$ and $w_2$ when $k_1\leq \ell_1<k_2\leq\ell_2$ and $k_2-\ell_1<|u|$.}
\label{Claim-CaseC}
\end{figure}

If  $|u|>|z_2|=k_2-\ell_1$ then a prefix of $u_1'$ maps to suffix of $u$. Because a prefix of $u$ maps also to prefix of $z_1$, $u_1'\neq u$, i.e.,
$u_1'= u^R$.  Similarly, because a suffix of $z_2z_3$ maps to a suffix of $u_2'$, and a prefix of $u^R$ maps to a suffix of $u$, it can't be $u=u_2'$ (symbols of $w$ are disjoint with those in $u$). Hence, $u_2'=u^R$. 
We can exclude the situation in  Figure \ref{Claim-CaseC}a where $f(u)$ is a prefix of $z_1z_2$. Thus it must be that
 $|z_1z_2|<|u|$, and a suffix of $u$, call it $v$, maps to a prefix of $u$.
Because $\Sigma[u]\cap\Sigma[z_3]=\emptyset$, it must be that the prefix $v^R$ of $u^R$ maps into a suffix of $u$.
Hence the symbols in a prefix of $u$ equal symbols in a suffix of $u$, which can only happen if $|z_1|=0$ contrary to the assumption that $k_1\neq k_2$. 
 Therefore $|u|\leq|z_2|$.
\end{proof}

The structure of the words that allow distinct sequential insertions yielding equivalent words is described with the 
following recursive construction. 
\begin{definition}\label{tangled}
Let $m\geq 0$ and $\nu,j\geq 1$ be integers and let $s_0=12\cdots m 12 \cdots m$ (if $m=0$, $s_0=\epsilon$).
Define $s_j=s_{j-1}\star\rho(\nu,|s_{j-1}|-m+1,|s_{j-1}|+1)$ for $j=1,2,\dots$. The word $s_j$ is called a {\em $\rho$-tangled cord 
at level $j$}, and is denoted $s_j=T_\rho(\nu,m,j)$.

Let $\nu$ divide $m$ and $t_0 =\text{Int}(h,\nu)$ where $h=\frac{m}{\nu}$ (if $m=0$, $s_0=\epsilon$).
Define $t_j=t_{j-1}\star\tau(\nu,|s_{j-1}|-m+1,|s_{j-1}|+1)$ for $j=1,2,\dots$. The word $t_j$ is called a {\em $\tau$-tangled cord 
at level $j$}, and is denoted $t_j=T_\tau(\nu,m,j)$.
\end{definition}

In both $\rho$- and $\tau$-tangled cord words the insertion in the subsequent word is performed at the end of the previous word (inserting $u'$) and $m$ symbols from the last symbol (inserting $u$).

\begin{example}\label{ex_tc}
A level 3 $\rho$-tangled cord $w=T_\rho(2,1,3)$ is obtained in the following way
$$ 
s_0 = 11\qquad
s_1 = 1\uRed{23}1\uRed{23}\qquad 
s_2 = 12312\uRed{45}3\uRed{45}\qquad
s_3 = 123124534\uRed{67}5\uRed{67},
$$
and a level 2 $\tau$-tangled cord $w=T_\tau(2,4,2)$ is obtained with
$$
t_0 = 12342143\qquad
t_1= 1234\uRed{56}2143\uRed{65}\qquad
t_2= 12345621\uRed{78}4365\uRed{87}.
$$
\end{example}

Recall that $\text{Int}(h,1)$ for $h\geq 1$ is a repeat word of size $h$.
It follows that $T_\rho(1,m,j)=T_\tau(1,m,j)$ for $m\geq 0$ and $j\geq 1$.
Moreover, if $w=T_\sigma(\nu,0,j)$ for $\sigma\in\{\rho,\tau\}$ and $\nu,j\geq 1$, then
for $|v_i|=\nu$,
$$
w=\begin{cases}
v_1v_1 \cdots v_j v_j & \text{if } \sigma=\rho\\
v_1v_1^R \cdots v_j v_j^R & \text{if } \sigma=\tau.
\end{cases}
$$

As with interleaving and nested insertions, we will use the ``image tracking'' method given by Lemmas \ref{lem_L-S} and \ref{lem_short} when considering sequential insertions.
Since both indices $k_1\leq\ell_1$ of the first insertion precede the indices $k_2\leq\ell_2$ of the second, we must adapt this method to describe the structure of $w$ in between $\ell_1$ and $k_2$.
In Lemma~\ref{lem_caseC1}, we begin by considering the case in which $k_1=\ell_1=1$.

\begin{lemma}\label{lem_caseC1}
If $\mc I_1$ and $\mc I_2$ are sequential insertions such that $k_1 = \ell_1 < k_2 \leq \ell_2$, then $k_2=\ell_2$ and $k_2-\ell_1=2p\nu$  for some positive integer $p$. 
Moreover, if $\mathcal{I}_1,\mathcal{I}_2\in\Rep$ ($\mathcal{I}_1,\mathcal{I}_2\in\Ret$ resp.) then $z_1z_2z_3 \sim T_{\rho}(\nu,0,p)$ ($z_1z_2z_3 \sim T_{\tau}(\nu,0,p)$ resp.).
\end{lemma}
\begin{figure}[!h]
\subfloat{}{\includegraphics[]{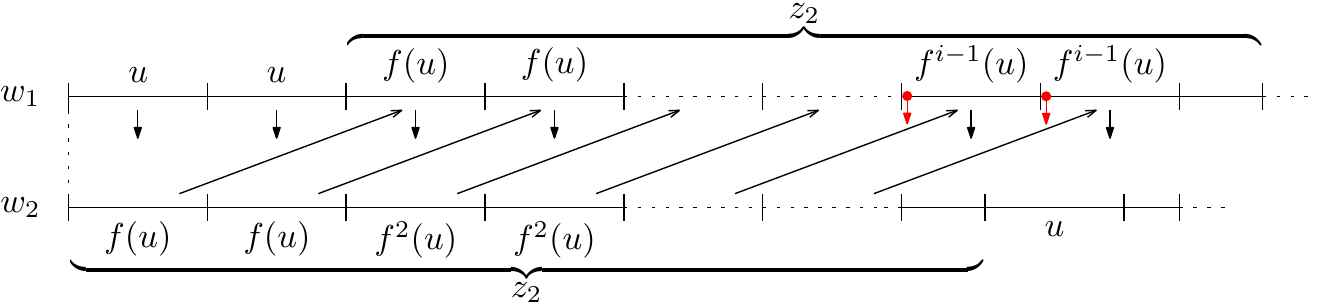}}

\vspace{3mm}

\subfloat{}{\includegraphics[]{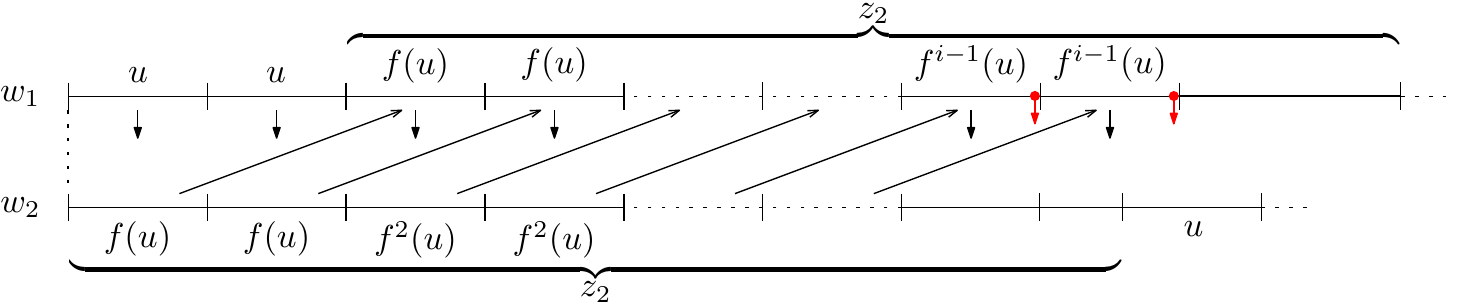}}
\caption{Representation of $w_1$ and $w_2$ when $|z_1|=0$ and $|z_2|$ is not an even multiple of $|u|$.} 
\label{C-noz2} 
\end{figure}
\begin{proof}
Suppose that $k_1 = \ell_1 < k_2 \leq \ell_2$. The situation is such that $w_1=uu'z_2z_3$ and $w_2=z_2uz_3u'$. Recall that $|z_2|\geq|u|=\nu$ by Lemma \ref{lem_z3}.
Assume that $|z_2|=k_2-\ell_1$ is not an even multiple of $|u|$, and let $u=b_1\cdots b_\nu$.
If $\mathcal{I}_1,\mathcal{I}_2\in\Rep$, then there exists a positive integer $i$ such that either $f^i(b_1)$
or $f^i(b_\nu)$ belongs to $\Sigma[z_2]\cap \Sigma[u]$ (see Figure \ref{C-noz2}) contradicting the fact that 
 $\Sigma[w]\cap\Sigma[u]=\emptyset$.
Similarly, if $\mathcal{I}_1,\mathcal{I}_2\in\Ret$, then there exists a positive integer $i$ such that $f^i(b_\nu)f^i(b_\nu)\sqsubseteq u$, or $f^i(b_1)\in\Sigma[z_2]\cap\Sigma[u]$. 
Then either $u\not\in\Sigma_{SOW}$ or $\Sigma[w]\cap\Sigma[u]\neq\emptyset$.
It follows that $k_2-\ell_1 = |z_2|=2p|u|=2p\nu$ for some positive integer $p$ such that $f^{p+1}(u)=u$.

We now apply Lemma~\ref{lem_L-S} with $s=uu'$, $z=z_2$, and $t=u\pre(z_3u',\nu)$.
Note that $|z_2|=2p\nu$ implies that $s=uu'=s_1$, and thus $u\pre(z_3u',\nu)=uu'$ (i.e., $z_3=\epsilon$) because $\Sigma[u]\cap\Sigma[w]=\emptyset$.
Thus $z_2=y_1y_1\cdots y_py_p$ ($z_2=y_1y_1^R\cdots y_py_p^R$ resp.) where $y_i=f^i(u)$ for $1\leq i\leq p$ in the case of two repeat (return resp.) insertions.
That is to say, $z_2$ is equivalent to $T_\rho(\nu,0,p)$ ($T_\tau(\nu,0,p)$ resp.).
\end{proof}

\begin{proposition}\label{prop_caseC2}
If $\mc I_1$ and $\mc I_2$ are sequential insertions such that $k_1 \le \ell_1 < k_2 \leq \ell_2$, then $k_2-\ell_1=2p\nu$  for some positive integer $p$ and $|z_1|=|z_3|=q$. 
Moreover, if $\mathcal{I}_1,\mathcal{I}_2\in\Rep$ then $z_1z_2z_3 \sim T_{\rho}(\nu,q,p)$, while if $\mathcal{I}_1,\mathcal{I}_2\in\Ret$ then $z_1z_2z_3 \sim T_{\tau}(\nu,q,p)$ where $\nu$ divides $q$.
\end{proposition}

\begin{proof}
The situation $k_{1} = \ell_{1}$
follows from Lemma \ref{lem_caseC1}, so we assume that $k_1 < \ell_1 < k_2 \leq \ell_2$ 
with $w_1=uz_1u'z_2z_3$, $w_2=z_1z_2uz_3u'$, and recall that $|z_2|\geq|u|=\nu$ by Lemma \ref{lem_z3}. 

(1) Suppose $\mathcal{I}_1,\mathcal{I}_2\in\Rep$; we consider two cases: (1.a) $0<|z_1|<|u|$ and (1.b) $|z_1|\geq |u|$.

(1.a) Let $0<|z_1|<|u|$, write $x_1=z_1$,
and let $y_1$ be the prefix of $z_2$ such that $|y_1|=|u|$; note that $y_1$ exists by Lemma~\ref{lem_z3} and that $uz_1\sim z_1y_1$.
Lemma \ref{lem_L-S} implies $u_1=v_1v_2$ with $v_1\neq\epsilon$, $z_1=f(v_1)=x_1$, and $y_1=f(v_2)f^2(v_1)=f(v_2)f(x_1)$.

Write $v=f(v_2)$ so that $y_1=vf(x_1)$.
Then, $f(uy_1)=x_1vf(y_1)=x_1vf(v)f^2(x_1)$.
Note that $x_1v\sqsubseteq z_2$, which implies that $y_1x_1v$ is a prefix of $z_2$ of length $2|u|$.
Moreover, $f(y_1)$ cannot have a proper factor in common with $u$, otherwise $f(x_1)\cap u\neq\emptyset$ contradicting the fact that $\Sigma[z_2]\cap \Sigma[u]=\emptyset$.
Thus, there are two possibilities: either $f(y_1)=u$ or $f(y_1)$ is a proper factor of $z_2$.
In the first situation,  $z_2=y_1x_1v=vf(x_1)x_1v$, and thus $|z_2|=2|u|=2\nu$, with 
$f(v)$  a prefix of $u$; 
thus, $z_3=f(x_1)$ and $|z_3|=|z_1|$.
Then $z_1z_2z_3=x_1y_1x_1vf(x_1)=x_1y_1x_1y_1$.
Observe that $z_1z_2z_3\sim T_\rho(\nu, \ell_1-k_1,1)$ and in particular $s_0\sim x_1x_1$ (where $s_0$ starts the induction of the tangled cord as in Definition~\ref{tangled}).

\begin{figure}[h!]
\centering
\includegraphics[scale=.75]{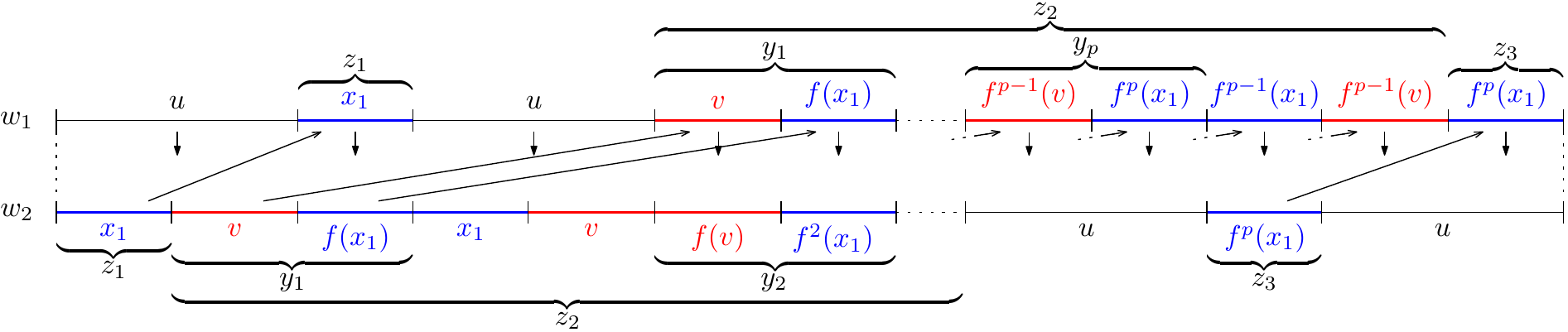}
\caption{Representation of $w_1$ and $w_2$ when $|z_1|<|u|$ and $f(y_1)$ is a proper factor of $z_2$. }\label{C-z1_short}
\end{figure}

In the second situation, we set  $y_2=f(y_1)=f(v)f^2(x_1)\sqsubseteq z_2$.
Similarly as above, we have that $y_1x_1vy_2f(x_1)f(v)$ is a prefix of $z_2$ of length $4|u|$, and either $f(y_2)=u$ or $f(y_2)$ is a proper factor of $z_2$. Inductively, 
 there is $p\geq 1$ such that $y_p=f^{p-1}(v)f^p(x_1)$ and $f(y_p)=u$ (see Figure \ref{C-z1_short}).
It follows that $$z_2=y_1x_1vy_2f(x_1)f(v)\cdots y_pf^{p-1}(x_1)f^{p-1}(v) \text{ with } p=\frac{k_2-\ell_1}{2\nu}.$$
Because $f^{p}(v)$ is a prefix of $u$, $z_3=f^p(x_1)$ and $|z_3|=|z_1|$.
As a consequence,
$$
z_1z_2z_3=x_1y_1x_1vy_2f(x_1)f(v)\cdots y_pf^{p-1}(x_1)f^{p-1}(v)f^p(x_1)=s_p
$$
for $p\geq 1$.
Observe that $z_1z_2z_3\sim T_\rho(\nu,\ell_1-k_1,p)$ because $y_i=f^{i-1}(v)f^i(x_1)$ for $1\leq i\leq p$.
In fact, taking $s_0 \sim x_1x_1$
we have 
$$
s_p= x_1y_1x_1vy_2f(x_1)f(v)\cdots y_pf^{p-1}(x_1)y_p\sim s_{p-1}\star \rho(\nu,|s_{p-1}|-|z_1|+1,|s_{p-1}|+1)\sim T_\rho(\nu,|z_1|,p).
$$

(1.b) Let $|z_1|\geq |u|=\nu$, and set as above $y_1$ to be the prefix of $z_2$ such that $|y_1|=|u|$.
Because $uz_1\sim z_1y_1$, by Lemma \ref{lem_L-S} we have that
$u=v_1v_2$ with $v_1\neq\epsilon$, 
$z_1=x_1x_2\cdots x_h$, where $f^i(u)=x_i$ for $1\leq i < h$ (if any such exist) with $x_h=f^h(v_1)$, and 
$y_1=f^h(v_2)f^{h+1}(v_1)=f^h(v_2)f(x_h)$.
Write $v=f^h(v_2)$ so that $y_1=vf(x_h)$; if $v_2=\epsilon$, then $v=\epsilon$ and $y_1=f(x_h)$ (see Figure \ref{C-z1_general}).
Similarly as in the case (1.a), using the same notation, it implies that there is a positive integer $p\geq 1$ such that $f(y_p)=u$.
\begin{figure}[h!]
\centering
\includegraphics[scale=.9]{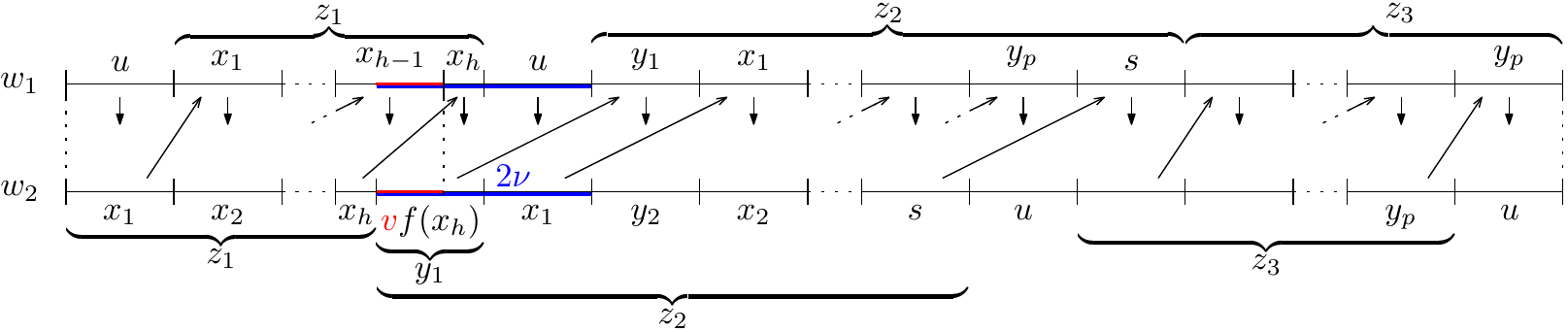}
\caption{Representation of $w_1$ and $w_2$ when $|z_1|\geq |u|$ and $f(y_1)$ is a proper factor of $z_2$.
If $p<h$, then $s=x_p$; otherwise $s=f^{p-h}(x_hv)$.}\label{C-z1_general}
\end{figure}

Suppose first $p<h$; then $z_2=y_1x_1\cdots y_px_p$, indicating $p=\frac{k_2-\ell_1}{2\nu}$ because $|y_i|=|x_i|$ for $1\leq i\leq p<h$ (see Figure \ref{C-z1_general} where $s=x_p$).
Because $y_px_p$ is a suffix of $z_2$ and $f(y_px_p)=ux_{p+1}$, it follows that $x_{p+1}\cdots x_h$ is a prefix of $z_3$.
Note that $f(x_{h-1}x_h)=x_hvf(x_h)=x_hy_1$; so $x_{p+1}\cdots x_hy_1$ is a prefix of $z_3$.
Inductively, we have that $z_3=x_{p+1}\cdots x_hy_1\cdots y_p$ because $f^{-1}(u)=y_p$, and so $|z_3|=|z_1|$ since $|y_i|=|u|=\nu$ for all $i$.
Hence,
$$z_1z_2z_3=x_1\cdots x_hy_1x_1\cdots y_px_px_{p+1}\cdots x_hy_1\cdots y_p$$
for $1\leq p < h$.
Observe that $z_1z_2z_3\sim T_\rho(\nu,\ell_1-k_1,p)$ in this case (i.e., when $1\leq p < h$).
Indeed, following Definition~\ref{tangled}, taking $s_0 \sim x_1\cdots x_h x_1\cdots x_h$
 we have that
\begin{align*}
s_1 &\sim x_1\cdots x_h y_1 x_1\cdots x_h y_1\\
s_2 &\sim x_1\cdots x_h y_1 x_1 y_2 x_2 \cdots x_h y_1 y_2\\
&\vdots\\
s_p &\sim x_1\cdots x_h y_1 x_1 y_2 x_2\cdots y_p x_p x_{p+1}\cdots x_h y_1\cdots y_p.
\end{align*}

Now suppose that $p\geq h$, then
$$
z_2=y_1x_1y_2x_2\cdots x_{h-1}y_hx_hvy_{h+1}f(x_h)f(v)y_{h+2}\cdots y_pf^{p-h}(x_h)f^{p-h}(v),
$$
where $y_i=f^{i-1}(v)f^i(x_h)$ for $1\leq i\leq p$ (see Figure \ref{C-z1_general}, where $s=f^{p-h}(x_hv)$).
Because $x_hv$ has length $|u|=\nu$, we again obtain that $|z_2|$ is a multiple of $2|u|=2\nu$; in particular $|z_2|=2p\nu$ where $p=\frac{k_2-\ell_1}{2\nu}$.
Because $f^{-1}(u)=y_p$, we have that
$$
z_3=f^{p-h+1}(x_h)\underbrace{f^{p-h+1}(v)f^{p-h+2}(x_h)}_{y_{p-h+2}} \underbrace{f^{p-h+2}(v)f^{p-h+3}(x_h)}_{y_{p-h+3}} \cdots \underbrace{f^{p-1}(v)f^{p}(x_h)}_{y_p}.$$
Note that $|z_3|=|z_1|$ since $|y_i|=|u|=\nu$ for all $i$. 
We observe that $z_1z_2z_3\sim T_\rho(\nu,\ell_1-k_1,p)$ in the case  $p\geq h$ as well.
Indeed, taking $s_0, \ldots, s_p$ as in Definition~\ref{tangled}:
\begin{align*}
s_0 &\sim x_1\cdots x_h x_1\cdots x_h\\
s_1 &\sim x_1\cdots x_h y_1 x_1\cdots x_h y_1\\
&\vdots\\
s_h &\sim x_1\cdots x_h y_1 x_1 y_2 x_2\cdots y_h x_h y_1 y_2\cdots y_h\\
s_{h+1}&\sim x_1\cdots x_h y_1 x_1 y_2 x_2\cdots y_h x_h v y_{h+1} f(x_h) y_2\cdots y_h y_{h+1}\\
&\vdots\\
s_p&\sim z_1z_2z_3.
\end{align*}

\sloppypar (2) Suppose that $\mathcal{I}_1,\mathcal{I}_2\in\Ret$.
First, observe that $|z_1|\geq |u|$.
On the contrary, suppose that $0<|z_1|<|u_1|$.
Similarly, and with the same notation, as in case (1.a), we have that
$$
z_2=y_1v^Rx_1^Ry_2f(v^R)f(x_1^R)\cdots y_pf^{p-1}(v^R)f^{p-1}(x_1^R) \text{ with } p=\frac{k_2-\ell_1}{2|u|},
$$
where $x_1=z_1$, $v$ is the nonempty prefix of $z_2$ such that $x_1v=f(u)$, $y_i=f^{i-1}(v)f^i(x_1)$ for $1\leq i\leq p$ and $f(y_p)=f(f^{p-1}(v)f^p(x_1))=f^p(v)f^{p+1}(x_1)=u$.
Note that $f^p(v)$ is a proper prefix of $u$, hence $f^p(v)^R$ is a proper suffix of $u^R$.
Considering the suffix $y_pf^{p-1}(v^R)f^{p-1}(x_1^R)$ of $z_2$, we have that $f(y_pf^{p-1}(v^R)) = f(y_p)f^p(v^R) = uf^p(v)^R \sqsubseteq w_2$ (see Figure \ref{C-z1_short}).
But $uf^p(v)^R \sqsubseteq w_2$ and $u^R \sqsubseteq w_2$ contradicts the fact that $w_2$ is a DOW; thus $|z_1|\geq |u|$.\\

In the rest of the proof we use  similar  arguments to those used in case (1.b) for repeat insertions (with the same notation).

If $p<h$, then $z_2=y_1x_1^R\cdots y_px_p^R$ with $f(y_p)=u$, which implies that $p=\frac{k_2-\ell_1}{2\nu}$ because $|y_i|=|x_i|$ for $1\leq i\leq p<h$.
It follows that
$$
x_{p+1}^R\cdots x_{h-1}^Rv^Rx_h^Rf(v^R)f(x_h^R)\cdots f^{p-1}(v^R)f^{p-1}(x_h^R)f^p(x_h^R)
$$
is a prefix of $z_3$ (see Figure \ref{C-z1_general}).
Assume that $\nu$ does not divide $|z_1|=\ell_1-k_1$ in this situation; then $v\neq\epsilon$.
Since $f(y_p)=f(f^{p-1}(v)f^p(x_h))=f^p(v)f^{p+1}(x_h)=u$, we have that $f^p(v)^R$ is a proper suffix of $u^R$ and $f^{p+1}(x_h)^R$ is a proper prefix of $u^R$.
But then $f(f^{p-1}(v^R)f^{p-1}(x_h^R)f^p(x_h^R))=f^p(v^R)f^p(x_h^R)f^{p+1}(x_h^R)\sqsubseteq w_2$ with $u^R \sqsubseteq w_2$ contradicts the fact that $w_2$ is a DOW.
Thus, 
$$
p< h = \left\lceil \frac{\ell_1-k_1}{\nu} \right\rceil = \frac{\ell_1-k_1}{\nu}.
$$
Then $|z_1|=h\nu$, $v=\epsilon$, $|x_h|=\nu$, and $y_i=f^i(x_h)$ for $1\leq i\leq p$.
Because $f^{-1}(u^R)=y_p^R$, it follows that $z_3=x_{p+1}^R\cdots x_h^Rf(x_h^R)\cdots f^p(x_h^R)=x_{p+1}^R\cdots x_h^Ry_1^R\cdots y_p^R$, and so $|z_3|=|z_1|=h\nu$ since $|y_i|=|u|$ for all $i$.
Hence,
$$z_1z_2z_3=x_1\cdots x_hy_1x_1^R\cdots y_px_p^Rx_{p+1}^R\cdots x_h^Ry_1^R\cdots y_p^R$$
for $1\leq p < h$.
Similarly as in case (1.b) we observe that $z_1z_2z_3\sim T_\tau(\nu,\ell_1-k_1,p)$ in this situation.

In case  $p\geq h$
we have that
$$
z_2=y_1x_1^Ry_2x_2^R\cdots x_{h-1}^Ry_hv^Rx_h^Ry_{h+1}f(v^R)f(x_h^R)y_{h+2}\cdots y_pf^{p-h}(v^R)f^{p-h}(x_h^R),
$$
where $y_i=f^{i-1}(v)f^i(x_h)$ for $1\leq i\leq p$ and $f(y_p)=u$.
Because $|v^Rx_h^R|=|u|$, it follows that $|z_2|$ is a multiple of $2|u|$; in particular $|z_2|=2p\nu$, where $p=\frac{k_2-\ell_1}{2\nu}$.
Moreover, we have that
$$
f^{p-h+1}(v^R)f^{p-h+1}(x_h^R)\cdots f^{p-1}(v^R)f^{p-1}(x_h^R)f^{p}(x_h^R)
$$
is a prefix of $z_3$ (see Figure \ref{C-z1_general}).
Similarly as above, $|u|=\nu$ divides $|z_1|=\ell_1-k_1$.
Thus $h=\frac{\ell_1-k_1}{\nu}$; so $|z_1|=h\nu$, $v=\epsilon$, $|x_h|=|u|$, and $y_i=f^i(x_h)$ for $1\leq i\leq p$.
It then follows that
\begin{align*}
z_2&=y_1x_1^Ry_2x_2^R\cdots x_{h-1}^Ry_hx_h^Ry_{h+1}f(x_h^R)y_{h+2}\cdots y_pf^{p-h}(x_h^R)\\
&=y_1x_1^Ry_2x_2^R\cdots x_{h-1}^Ry_hx_h^Ry_{h+1}y_1^Ry_{h+2}\cdots y_py_{p-h}^R\\
z_3&=f^{p-h+1}(x_h^R)f^{p-h+2}(x_h^R)\cdots f^{p}(x_h^R)=y_{p-h+1}^Ry_{p-h+2}^R\cdots y_p^R
\end{align*}
because $f^{-1}(u^R)=y_p^R$.
Note that $|z_3|=|z_1|=h\nu$ since $|y_i|=|u|$ for all $i$. Finally we observe that $z_1z_2z_3\sim T_\tau(\nu,\ell_1-k_1,p)$.
\end{proof}

\begin{example}\label{ex_c}
Let $w = 12312453467567 = T_\rho(2,1,3)$ and consider the insertions $\rho(2,1,2)$ and $\rho(2,14,15)$ into $w$; also, let $w' = 123456214365 = T_\tau(2,4,1)$ and consider the insertions $\tau(2,1,5)$ and $\tau(2,9,13)$ into $w'$. Then
\begin{align*}
w_1 &= \uRed{89}1\uRed{89}2312453467567
\qquad\qquad
w_1' = \uRed{78}1234\uRed{87}56214365\\
w_2 &= 1231245346756\uRed{89}7\uRed{89}
\qquad\qquad
w_2' = 12345621\uRed{78}4365\uRed{87}
\end{align*}
Note that $w_1\sim w_2$ and $w_1'\sim w_2'$.
\end{example}

\begin{corollary}\label{cor_C-conv}
For every $T_\sigma(\nu,m,j)$, we have that $T_\sigma(\nu,m,j)\star\mathcal{I}(\nu,1,m+1) \sim T_\sigma(\nu,m,j+1)$ where $\sigma\in\{\rho,\tau\}$ and $\mathcal{I}\in\Rep$ ($\mathcal{I}\in\Ret$ resp.) if $\sigma=\rho$ ($\sigma=\tau$ resp.).
\end{corollary}
\begin{proof}
The result follows by using similar arguments as in the proofs of Lemma~\ref{lem_caseC1} and Proposition~\ref{prop_caseC2}.
\end{proof}

A DOW is a \emph{palindrome} if it is equivalent to its reverse \cite{Burns2013}.
Let $m\geq 1$ be an integer; observe that the repeat word $12\cdots m 12\cdots m$ and the return word $12\cdots m m \cdots 21$ are palindromes.
Let $h,\nu\geq 1$.
Observe that $\text{Int}(h,\nu) = x_1 \cdots x_h x_1^R \cdots x_h^R$ is a palindrome.
Indeed, consider $f:\Sigma\to\Sigma$ such that $f(x_i)=x_{h-i+1}$ for all $i$.
Similarly, $\text{Nes}(h,\nu) = x_1 \cdots x_h x_h \cdots x_1$ is a palindrome; consider $f:\Sigma\to\Sigma$ such that $f(x_i)=x_i^R$ for all $i$.

\begin{proposition}\label{prop_palind}
Every $T_\sigma(\nu,m,j)$ is a palindrome where $\sigma\in\{\rho, \tau\}$.
\end{proposition}
\begin{proof}
Let $\nu\geq 1$, $m\geq 0$, and $\sigma\in\{\rho,\tau\}$.
We show that the result holds by inducting on $j$.
Note that $T_\rho(\nu,m,1)$ is a repeat word of size $m+\nu$ and $T_\tau(\nu,m,1) = \text{Int}(\frac{m}{\nu}+1,\nu)$.
In either case, $T_\sigma(\nu,m,1)$ is a palindrome.
Next, suppose that $w=T_\sigma(\nu,m,j-1)$ is a palindrome for $j>1$, and let $n=|w|$.
The following shows that $T_\rho(\nu,m,j)$ is a palindrome:
\vskip.5\baselineskip
\centerline{
\begin{tabular}{rll}
$T_\rho(\nu,m,j)$ & $\sim w\star\rho(\nu,1,m+1)$ & (by Corollary~\ref{cor_C-conv})\\
& $\sim w^R\star\rho(\nu,1,m+1)$ & (because $w\sim w^R$)\\
& $\sim \left(w\star\rho(\nu,n-m+1,n+1)\right)^R$ & \\
& $\sim \left(T_\rho(\nu,m,j)\right)^R$ & (by the definition of $T_\rho(\nu,m,j)$).
\end{tabular}
}
\vskip.5\baselineskip
The case for $T_\tau(\nu,m,j)$ follows equivalently.
\end{proof}

Observe that there exist words that are palindromes but are neither $\rho$-tangled cords nor $\tau$-tangled cords, like $12324143$.
However, note that $12324143$ is equivalent to a cyclic permutation of the tangled cord $T_\rho(1, 1, 3)$.

\begin{definition}
Let $w\in\Sigma_{DOW}$, and let $uu'$ be a repeat (return resp.) word in $w$.
We say that $uu'$ is a {\em maximal} repeat (return resp.) word in $w$ if for any repeat (return resp.) word $vv'$ in $w$ such that $u\sqsubseteq v$, we have that $u=v$.
\end{definition}

\begin{example}
Consider $T_\mathcal{\rho}(2,3,2)=s_2$ and $T_\mathcal{\tau}(2,4,1)=s_1'$.
\vskip.5\baselineskip
\centerline{
\begin{tabular}{ll}
$s_0 = 123123$ & $s_0' = 12342143$\\
$s_1 = s_0\star\rho(2,4,7) = 123\uRed{45}123\uRed{45}$ & 
$s_1' = s_0'\star\tau(2,5,9) = 1234\uRed{56}2143\uRed{65}$\\
$s_2 = s_1\star\rho(2,8,11) = 1234512\uRed{67}345\uRed{67}$
\end{tabular}
}
\vskip.5\baselineskip
Note that $T_\mathcal{\rho}(2,3,2)$ contains four maximal repeat words: $33$, $1212$, $4545$, and $6767$.
On the other hand, $T_\mathcal{\tau}(2,4,1)$ contains three maximal return words: $1221$, $3443$ and $5665$.
\end{example}

As the previous example suggests, every $\tau$-tangled cord ends with a maximal return word of size $\nu$.
The same property holds for $\rho$-tangled cord whenever $j>1$.

\begin{proposition}
The following hold.
\begin{enumerate}[label={(\arabic*)}]
\setlength{\itemsep}{-3pt}
\item If $T_\mathcal{\rho}(\nu_1,m_1,j_1)\sim T_\mathcal{\rho}(\nu_2,m_2,j_2)$ and $j_1>1$, then $(\nu_1, m_1, j_1)=(\nu_2, m_2, j_2)$.
\item If $T_\mathcal{\tau}(\nu_1,m_1,j_1)\sim T_\mathcal{\tau}(\nu_2,m_2,j_2)$, then $(\nu_1, m_1, j_1)=(\nu_2, m_2, j_2)$.
\end{enumerate}
\end{proposition}
\begin{proof}
(1) Let $T_\mathcal{\rho}(\nu_1,m_1,j_1)$ and $T_\mathcal{\rho}(\nu_2,m_2,j_2)$ be as given with $j_1>1$.
Because these DOWs are equivalent and thus have equal lengths, note that $m_1+\nu_1j_1=m_2+\nu_2j_2$.
Moreover, $T_\mathcal{\rho}(\nu_1,m_1,j_1)$ ends with a maximal repeat word of size $\nu_1$.
If $j_2=1$, then $T_\mathcal{\rho}(\nu_2,m_2,j_2)$ is a repeat word of size $m_2+\nu_2$ which cannot be equivalent to $T_\mathcal{\rho}(\nu_1,m_1,j_1)$; thus, $j_2>1$.
It follows that $T_\mathcal{\rho}(\nu_2,m_2,j_2)$ ends with a maximal repeat word of size $\nu_2$; so $\nu_1=\nu_2$.
By setting $\nu=\nu_1$, we get $|m_1-m_2| = \nu|j_1-j_2|$.
We now distinguish between the following two cases:
\vskip.5\baselineskip
\centerline{
\begin{tabular}{cc}
(i) $m_1,m_2>0$ & (ii) $m_1=0$ or $m_2=0$.
\end{tabular}
}
\vskip.5\baselineskip

\sloppypar (1.i) Suppose $m_1,m_2>0$.
Then $\pre(T_\mathcal{\rho}(\nu,m_1,j_1), m_1+\nu+1)=12\cdots m_1v_11$ and $\pre(T_\mathcal{\rho}(\nu,m_2,j_2), m_2+\nu+1)=12\cdots m_2\hat{v}_11$, with $|v_1|=|\hat{v}_1|=\nu$.
It follows that $m_1=m_2$ because $T_\mathcal{\rho}(\nu,m_1,j_1)\sim T_\mathcal{\rho}(\nu,m_2,j_2)$.

\sloppypar (1.ii) Suppose that $m_1=0$ but $m_2\neq 0$.
Then $\pre(T_\mathcal{\rho}(\nu,m_1,j_1), 2\nu)=v_1v_1$ while $\pre(T_\mathcal{\rho}(\nu,m_2,j_2), m_2+\nu)=12\cdots m_2\hat{v}_1$, with $|v_1|=|\hat{v}_1|=\nu$.
But this contradicts the assumption that  $T_\mathcal{\tau}(\nu,m_1,j_1)\sim T_\mathcal{\tau}(\nu,m_2,j_2)$.
So $m_1=0$ implies $m_2=0$.
The same argument can be applied to show that $m_2=0$ implies $m_1=0$.

In both cases above, $m_1=m_2$ and so $j_1=j_2$.
We can apply a similar argument to prove (2).
\end{proof}

\subsection{Repeat vs Return Insertions}\label{ss_vs}
In this section we show that two distinct nontrivial insertions which yield equivalent words must both be of the same type, that is, both repeat insertions or both return insertions.
We note that the proof for this result follows the same methodology, notation,  and very similar arguments to those presented in Sections \ref{ss_A+B} and \ref{ss_C}.

\begin{theorem}\label{thm_repret}
\sloppypar Let $\mathcal{I}_1(\nu,k_1,\ell_1)$ and $\mathcal{I}_2(\nu,k_2,\ell_2)$ be two distinct nontrivial insertions into $w\in\Sigma_{DOW}$.
If $w_1=w\star\mathcal{I}_1(\nu,k_1,\ell_1)\sim w\star\mathcal{I}_2(\nu,k_2,\ell_2)=w_2$, then
either $\mathcal{I}_1,\mathcal{I}_2\in\Rep$ or $\mathcal{I}_1,\mathcal{I}_2\in\Ret$.
\end{theorem}

\begin{proof}
Let $w=z_1z_2z_3$, as in Sections \ref{ss_A+B} and \ref{ss_C} be of length $n$, and  $w_1\sim w_2$ with equivalence map $f$.
Because $\mc I_1$ and $\mc I_2$ are nontrivial, $\nu\ge 2$.
We prove the result by contradiction.

\begin{figure}[h!]
\begin{minipage}{.2\textwidth}
\centering
\bf{(a)}\includegraphics[scale=.8]{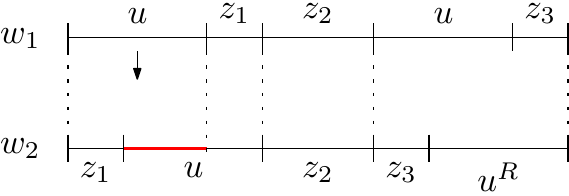}
\end{minipage}
\begin{minipage}{1\textwidth}
\centering
\bf{(b)}\includegraphics[scale=.8]{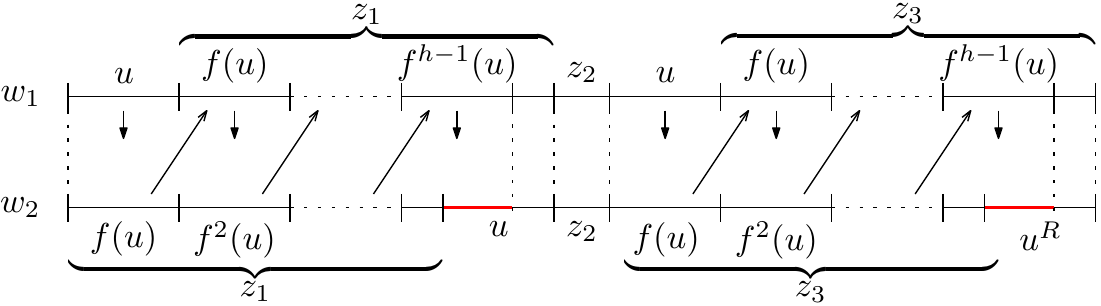}
\end{minipage}
\caption{Representation of $w_1$ and $w_2$ when $\mathcal{I}_1\in\Rep$, $\mathcal{I}_2\in\Ret$; (\textbf{a}) $0<|z_1|<|u|$, or (\textbf{b}) $|z_1|\geq |u|$.}
\label{A-RepRet} 
\end{figure}

{\it Case of interleaving or nested insertions.}  Suppose $k_1 < k_2 \leq \ell_1 < \ell_2$ and without loss of generality we assume that $\mathcal{I}_1\in\Rep$ and $\mathcal{I}_2\in\Ret$.
If $|z_1|<|u|$, then $f(u)\cap u\neq\emptyset$ (see Figure \ref{A-RepRet}a).
Because $uu$ is a repeat word in $w_1$, $f(u)f(u)$ is a repeat word in $w_2$; so by Lemma \ref{lem_comp} we have $|f(u)|=1$ or $|u|=1$ because $uu^R$ is return word in $w_2$.
Thus $|u|=1$ implying that the insertions must be trivial. 
If $|z_1|\geq|u|$, then observe that $f(u)$ is a prefix of $z_1$.
But then $f^2(u)$ is either a factor of $z_1$ or $f^2(u)\cap u\neq\emptyset$.
Inductively, for some $h>1$ we have that $f^i(u)$ is a factor of $z_1$ for $1\leq i<h$ and $f^h(u)\cap u\neq\emptyset$ (see Figure \ref{A-RepRet}b).
Consider the second occurrence of $u$ in $w_1$; we have that $f^i(u)$ is a factor of $z_3$ for $1\leq i<h$ and $f^h(u)\cap u^R\neq\emptyset$.
But $f^i(u)f^i(u)$ is a repeat word in $w_2$ of size $|u|$ for every $1\leq i\leq h$; by Lemma~\ref{lem_comp}, $f^h(u)\cap u\neq\emptyset$ implies that $|f^h(u)|=1$ or $|u|=1$ because $uu^R$ is a return word in $w_2$. Thus, $|u|=1$ contradicting $\nu\ge 2$.
The case of nested insertions, that is, $k_1 < k_2 \leq \ell_2 < \ell_1$, follows by similar arguments.

{\it Sequential insertions.} Suppose that $k_1 \leq \ell_1 < k_2 \leq \ell_2$ and  without loss of generality
assume that $\mathcal{I}_1\in\Rep$ and $\mathcal{I}_2\in\Ret$.
As in Lemma \ref{lem_caseC1} and Proposition \ref{prop_caseC2}, we consider the  cases: (I) $k_1=\ell_1$ and (II) $k_1\neq\ell_1$.
In both cases recall that $|z_2|\geq|u|$ by Lemma \ref{lem_z3}.
\begin{figure}[h]
\centering
 \includegraphics[]{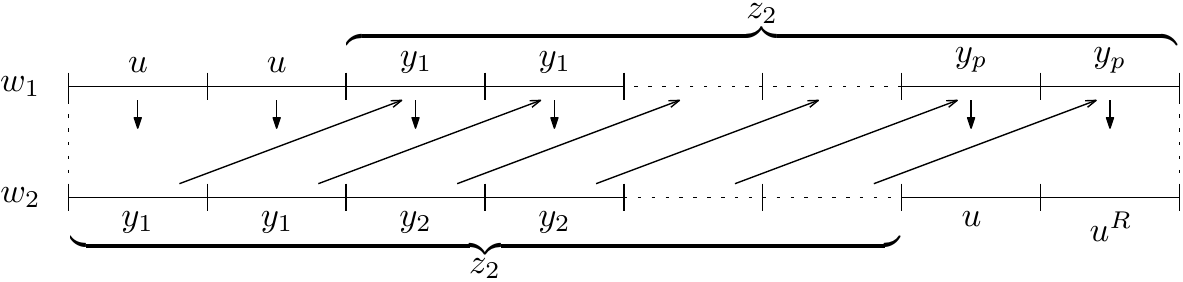}
  \caption{Representation of $w_1$ and $w_2$ when $\mathcal{I}_1\in\Rep$, $\mathcal{I}_2\in\Ret$, and $k_1 = \ell_1$.}
  \label{C-noz2-z3-RepRet} 
\end{figure}
(I) Let $k_1=\ell_1$; i.e., $z_1=\epsilon$.
As in proof of Lemma \ref{lem_caseC1} for repeat insertions, we have that $z_2=y_1y_1\cdots y_py_p$ where $y_i=f^i(u)$ for $1\leq i\leq p$, $f(y_p)=u$ and $p=\frac{k_2-\ell_1}{2\nu}$ (see Figure \ref{C-noz2-z3-RepRet}).
It follows that $f(y_p)f(y_p)$ is a repeat word in $w_2$ of size $|u|$.
Then, by Lemma \ref{lem_comp}, $f(y_p)=u$ implies that $|f(y_p)|=1$ or $|u|=1$, and the insertions must be trivial. 

\noindent
(II) Let $k_1\neq\ell_1$, so that $z_1\neq\epsilon$.
If $0<|z_1|<|u|$,  by  the proof of Proposition \ref{prop_caseC2} for case (1.a), we have that
$$
z_2=y_1x_1vy_2f(x_1)f(v)\cdots y_pf^{p-1}(x_1)f^{p-1}(v) \text{ with } p=\frac{k_2-\ell_1}{2\nu},
$$
where $x_1=z_1$, $v$ is the nonempty prefix of $z_2$ such that $x_1v=f(u)$, $y_i=f^{i-1}(v)f^i(x_1)$ for $1\leq i\leq p$ and $f(y_p)=f(f^{p-1}(v)f^p(x_1))=f^p(v)f^{p+1}(x_1)=u$ (see Figure \ref{C-z1_short}).
It follows that $f(y_pf^{p-1}(x_1)f^{p-1}(v))=uf^p(x_1)f^p(v)$, where $f^p(x_1)$ is a proper suffix of $y_p$ and $f^p(v)$ is a proper prefix of $u$.
Because $u^R\sqsubseteq w_2$, this contradicts the fact that $w_2$ is a DOW.

So assume that $|z_1|\geq |u|$, and let $y_1$ be the prefix of $z_2$ such that $|y_1|=|u|$ (exists by Lemma~\ref{lem_z3}).
By applying arguments similar to the ones used in the proof of Proposition \ref{prop_caseC2} for case (1.b) (with the same notation), we obtain that there exists a positive integer $p\geq 1$ such that $f(y_p)=u$.
If $p<h$, then $z_2=y_1x_1\cdots y_px_p$ and $z_3=x_{p+1}\cdots x_hy_1\cdots y_p$ (see Figure \ref{C-z1_general}).
On the other hand, for $p\geq h$ we have 
$$
z_2=y_1x_1y_2x_2\cdots x_{h-1}y_hx_hvy_{h+1}f(x_h)f(v)y_{h+2}\cdots y_pf^{p-h}(x_h)f^{p-h}(v)
$$
and $z_3=f^{p-h+1}(x_h)y_{p-h+2}\cdots y_p$ (see Figure \ref{C-z1_general}).
In both cases, $p=\frac{k_2-\ell_1}{2\nu}$ and $f(y_p)f(y_p)$ is a repeat word in $w_2$ of size $|u|$.
Then, by Lemma \ref{lem_comp}, $f(y_p)=u$ implies that $|f(y_p)|=1$ or $|u|=1$
contradicting the fact that the insertions are nontrivial.
\end{proof}

\section{Summary}\label{s_sum}

In this section, we summarize the results from Section~\ref{s_equiv}.
First, we collect the main results for Section~\ref{s_equiv}, generalizing appropriately.
Let $w\in\Sigma_{DOW}$ be in ascending order of length $n$ and $\mathcal{I}_i(\nu,k_i,\ell_i)$ for $i=1,2$ be two distinct insertions into $w$.
Assume that $k_1< k_2$ without loss of generality.
Write $w=z_0z_1z_2z_3z_4$ 
 and assume that the four locations of insertions (two from each $\mc I_1$ and $\mc I_2$) appear between words $z_i$ and $z_{i+1}$ for $i=0,\ldots,3$. In particular $|z_0|=k_1-1$.
In the following result we summarize Propositions \ref{prop_caseA}, \ref{prop_caseB}, and \ref{prop_caseC2}.
Observe that, by Theorem \ref{thm_repret}, the classification below is exhaustive.

\begin{theorem}\label{thm_main}
Let $w \in \Sigma_{DOW}$, and let $\mathcal{I}_i(\nu, k_{i}, \ell_{i})$ for $i=1,2$ be two distinct insertions into $w$.
Let $w_{i} = w \star \mathcal{I}(\nu, k_{i}, \ell_{i})$ for $i = 1, 2$.
Write $w=z_0z_1z_2z_3z_4$ using the notation above.
If $w_{1} \sim w_{2}$, then the following hold:
\vskip.5\baselineskip
\centerline{
\begin{tabular}{c|c|c}
& $\mathcal{I}_1,\mathcal{I}_2\in\Rep$ & $\mathcal{I}_1,\mathcal{I}_2\in\Ret$\\
\hline
Interleaving & \multirow{2}{*}{$z_1z_3$ is a repeat word} & \multirow{2}{*}{$z_1z_3\sim\text{Int}\left(\frac{|z_1|}{\nu}, \nu\right)$}\\
$k_1<k_2\leq\ell_1<\ell_2$ &  & \\
\hline
Nested & \multirow{2}{*}{$z_1z_3\sim\text{Nes}\left(\frac{|z_1|}{\nu},\nu\right)$} & \multirow{2}{*}{$z_1z_3$ is a return word}\\
$k_1<k_2\leq\ell_2<\ell_1$ &  & \\
\hline
Sequential & \multirow{2}{*}{$z_1z_2z_3\sim T_\rho\left(\nu,|z_1|,\frac{|z_2|}{2\nu}\right)$}  & \multirow{2}{*}{$z_1z_2z_3\sim T_\tau\left(\nu,|z_1|,\frac{|z_2|}{2\nu}\right)$}\\
$k_1\leq\ell_1<k_2\leq\ell_2$ &  & 
\end{tabular}
}
\vskip.5\baselineskip
\end{theorem}

\begin{remark}\label{rem_struc}
Let $w,w_1,w_2 \in \Sigma_{DOW}$ be as in Theorem~\ref{thm_main}; then the structure of the words $w_1$ and $w_2$ obtained by the insertions satisfies the following properties ($uu'$ is the inserted repeat or return word):
\vskip\baselineskip
\centerline{
\begin{tabular}{c|c|c}
& $\mathcal{I}_1,\mathcal{I}_2\in\Rep$ & $\mathcal{I}_1,\mathcal{I}_2\in\Ret$\\
\hline
\multirow{2}{*}{Interleaving} & \multirow{2}{*}{$z_1uz_3u$ is a repeat word} & \multirow{2}{*}{$z_1uz_3u^R\sim\text{Int}\left(\frac{|z_1|}{\nu}+1, \nu\right)$}\\
& & \\
\hline
\multirow{2}{*}{Nested} & \multirow{2}{*}{$z_1uuz_3\sim\text{Nes}\left(\frac{|z_1|}{\nu}+1,\nu\right)$} & \multirow{2}{*}{$z_1uu^Rz_3$ is a return word}\\
& & \\
\hline
\multirow{2}{*}{Sequential} & \multirow{2}{*}{$z_1z_2uz_3u\sim T_\rho\left(\nu,|z_1|,\frac{|z_2|}{2\nu}+1\right)$} & \multirow{2}{*}{$z_1z_2uz_3u^R\sim T_\tau\left(\nu, |z_1|,\frac{|z_2|}{2\nu}+1\right)$}\\
& &
\end{tabular}
}
\end{remark}

We generalize Proposition~\ref{prop_AB-conv} and Corollary~\ref{cor_C-conv} with the following theorem.

\begin{theorem}\label{thm_conv}
Let $w\in\Sigma_{DOW}$ and $\nu\in\mathbb{N}$.
Suppose $z_1z_2z_3\sqsubseteq w$ for some $z_1,z_3\in\Sigma_{SOW}$ and $z_2\in\Sigma^\ast$.
If one of the following (1) - (6) holds, then there exist two distinct insertions, $\mathcal{I}_1(\nu,k_1,\ell_1)$ and $\mathcal{I}_2(\nu,k_2,\ell_2)$, such that $w\star\mathcal{I}_1(\nu,k_1,\ell_1) \sim w\star\mathcal{I}_2(\nu,k_2,\ell_2)$.
	\vskip.5\baselineskip
    \begin{tabular}{ll}
	$(1)\quad z_1z_3$ is a repeat word in $w$ & ($\mathcal{I}_1,\mathcal{I}_2\in\Rep$ and interleaving insertions)\\
    $(2)\quad z_1z_3\sim\text{Int}(h,\nu)$ for some $h\geq 1$ & ($\mathcal{I}_1,\mathcal{I}_2\in\Ret$ and interleaving insertions)\\
    $(3)\quad z_1z_3\sim\text{Nes}(h,\nu)$ for some $h\geq 1$ & ($\mathcal{I}_1,\mathcal{I}_2\in\Rep$ and nested insertions)\\
    $(4)\quad z_1z_3$ is a return word in $w$ & ($\mathcal{I}_1,\mathcal{I}_2\in\Ret$ and nested insertions)\\
    $(5)\quad z_1z_2z_3\sim T_\mathcal{\rho}(\nu,m,j)$ &
    ($\mathcal{I}_1,\mathcal{I}_2\in\Rep$ and sequential insertions)\\
    $\quad\quad$ for some $m\geq 0$ and $j\geq 1$\\
    $(6)\quad z_1z_2z_3\sim T_\mathcal{\tau}(\nu,m,j)$ & 
    ($\mathcal{I}_1,\mathcal{I}_2\in\Ret$ and sequential insertions)\\
    $\quad\quad$ for some $m\geq 0$ and $j\geq 1$\\
	\end{tabular}
    \vskip.5\baselineskip
\end{theorem}

\begin{remark}
Observe that $11$ is the class representative for trivial repeat and return words in $\Sigma^\ast$, and 
$$
\text{Int}(1,1)=\text{Nes}(1,1)=T_\rho(1,0,1)=T_\tau(1,0,1)=11.
$$
By Theorem~\ref{thm_conv}, it follows that for any $w\in\Sigma_{DOW}\setminus\{\epsilon\}$, any $a\in\Sigma[w]$ can serve as $z_1=a$ and $z_3=a$, 
and for 
any $\nu\in\mathbb{N}$,
there exist a pair of interleaving (nested resp.) insertions into $w$, $\rho(\nu,k_1,\ell_1)$ and $\rho(\nu,k_2,\ell_2)$ ($\tau(\nu,k_1,\ell_1)$ and $\tau(\nu,k_2,\ell_2)$ resp.),  which yield equivalent words.
Moreover, if $aa\sqsubseteq w$, then there is also a pair of sequential and trivial insertions into $w$ which yield equivalent words.
\end{remark}

\section{Conclusion}
In this paper, we characterized the structure of a DOW $w$ which allows two distinct insertions to yield equivalent words.
In particular, we proved that a nontrivial repeat insertion and a nontrivial return insertion into $w$ cannot produce equivalent resulting words.
We summarized these results in Section~\ref{s_sum}; with them, we can consider a ``word graph'' in which vertices are equivalence classes of DOWs (with respect to $\sim$) and directed edges exist from $[w]_\sim$ to $[w']_\sim$ if $w'\sim w\star\mathcal{I}(\nu,k,\ell)$ for some insertion $\mathcal{I}(\nu,k,\ell)$.
Our results characterize when two insertions define the same edge in this graph. 
It is also natural to define the distance between $[w]_\sim$ and $[w']_\sim$ as the length of a shortest path from $[w]_\sim$ to $[w']_\sim$ in the word graph.
The notion of determining the distance from the empty word $[\epsilon]_\sim$ to $[w]_\sim$ may describe the complexity of a particular DNA rearrangement processes in the ciliate {\em Oxytricha trifallax} \cite{Burns2016,Jonoska2017}.
It may also be of interested to compare different paths between two given equivalence classes and to characterize the subgraphs which may appear in such a word graph.\\

{\bf Acknowledgement.}
This work is supported by NIH grant R01 GM109459, Simons Foundation grant 594594, and NSF grants CCF-1526485, DMS-1800443, and DMS-1764406.

\bibliographystyle{fundam}
\bibliography{citations}

\end{document}